\documentclass[11pt]{article}
\usepackage{amssymb}
\usepackage{amsmath}
\usepackage{amsfonts}
\usepackage{epic}
\usepackage{pstricks}
\usepackage{pst-node}
\usepackage{pstricks-add}
\usepackage{multicol}
\usepackage[normalem]{ulem}
\usepackage[pagewise, mathlines]{lineno}

\newcommand*\patchAmsMathEnvironmentForLineno[1]{%
  \expandafter\let\csname old#1\expandafter\endcsname\csname #1\endcsname
  \expandafter\let\csname oldend#1\expandafter\endcsname\csname end#1\endcsname
  \renewenvironment{#1}%
     {\linenomath\csname old#1\endcsname}%
     {\csname oldend#1\endcsname\endlinenomath}}%
\newcommand*\patchBothAmsMathEnvironmentsForLineno[1]{%
  \patchAmsMathEnvironmentForLineno{#1}%
  \patchAmsMathEnvironmentForLineno{#1*}}%
\AtBeginDocument{%
\patchBothAmsMathEnvironmentsForLineno{equation}%
\patchBothAmsMathEnvironmentsForLineno{align}%
\patchBothAmsMathEnvironmentsForLineno{flalign}%
\patchBothAmsMathEnvironmentsForLineno{alignat}%
\patchBothAmsMathEnvironmentsForLineno{gather}%
\patchBothAmsMathEnvironmentsForLineno{multline}%
}

\newcommand{\comment}[1]{}
\newcommand{\ora}{\overrightarrow}

\newtheorem{theorem}{Theorem}[section]

\newtheorem{lemma}[theorem]{Lemma}

\newtheorem{problem}{Problem}
\newenvironment{proof}{\noindent {\bf Proof.}}{$\Box$\\}

\begin{document}

\begin{center}
{\Large Directed domination in oriented hypergraphs}
\mbox{}\\[8ex]

\begin{multicols}{2}

Yair Caro\\[1ex]
{\small Dep. of Mathematics\\
University of Haifa-Oranim\\
Tivon 36006, Israel\\
yacaro@kvgeva.org.il}

\columnbreak

Adriana Hansberg\\[1ex]
{\small Instituto de Matem\'aticas\\
UNAM Juriquilla\\
Quer\'etaro, Mexico\\
ahansberg@im.unam.mx}\\[2ex]

\end{multicols}

\mbox{}\\[3ex]

This paper is dedicated to Lutz Volkmann\\
 on the occasion of his 75th birthday.\\[4ex]

\end{center}

\begin{abstract}
Erd\H{o}s [On Sch\"utte problem, Math. Gaz. 47 (1963)] proved that every tournament on $n$ vertices has a directed dominating set of at most $\log (n+1)$ vertices, where $\log$ is the logarithm to base $2$. He also showed that there is a tournament on $n$ vertices with no directed domination set of cardinality less than $\log n - 2 \log \log n + 1$. This notion of directed domination number has been generalized to arbitrary graphs by Caro and Henning in [Directed domination in oriented graphs, Discrete Appl. Math. (2012) 160:7--8.]. However, the generalization to directed r-uniform hypergraphs seems to be rare. Among several results, we prove the following upper and lower bounds on $\ora{\Gamma}_{r-1}(H(n,r))$, the upper directed $(r-1)$-domination number of the complete $r$-uniform hypergraph on $n$ vertices $H(n,r)$, which is the main theorem of this paper:\\
\[c (\ln n)^{\frac{1}{r-1}} \le \ora{\Gamma}_{r-1}(H(n,r)) \le C \ln n,\]
where $r$ is a positive integer and $c= c(r) > 0$ and $C = C(r) > 0$ are constants depending on $r$.\\

\end{abstract}

\section{Introduction}
Erd\H{o}s \cite{Er} proved that every tournament on $n$ vertices has a directed dominating set of at most $\log (n+1)$ vertices, where $\log$ is the logarithm to base $2$. He also showed that there is a tournament on $n$ vertices with no directed domination set of cardinality less than $\log n - 2 \log \log n + 1$. This notion of directed domination number has been generalized to arbitrary graphs by Caro and Henning \cite{CaHe1, CaHe2} and was recently treated in \cite{HLNT}. A generalization to directed domination in directed $r$-uniform hypergraphs seems to be rare \cite{KoSu}. For results on domination in hypergraphs, see \cite{Ach1, Ach2, JoTu}.\\

We consider an $r$-uniform hypergraph $H = (V,\mathcal{E})$, where $V$ is the vertex set and $\mathcal{E}$ is the edge set consisting of $r$-subsets of $V$. An {\it orientation} $D$ of $H$ is an oriented hypergraph $D=(V, \mathcal{E}(D))$, where the edge set $\mathcal{E}(D)$ consists of all edges in $\mathcal{E}$ such that each of them is provided with a linear ordering of its elements (exactly one of the $r!$ possible orders). \\

Let $D$ be an orientation of $H = H(V, \mathcal{E})$. For $1 \le p \le r-1$, a set $S \subseteq V$  is called a {\it directed $p$-dominating set} of $D$ if for every vertex $u \in V \setminus S$ there is an ordered edge $E \in \mathcal{E}(D)$ with $u \in E$ such that the first $p$ vertices from $E$ are in $S$. Observe that the definition covers the cases when $H(V, \mathcal{E})$ contains isolated vertices or even when $H$ is edgeless, in which case we must have $S = V$. A {\it minimum directed $p$-dominating set} $S$ is a directed $p$-dominating set of $D$ of minimum cardinality and $\ora{\gamma}_p(D) = |S|$. We denote with $\ora{\gamma}_p(H)$ and $\ora{\Gamma}_p(H)$ the minimum and, respectively, the maximum of $\ora{\gamma}_p(D)$, where $D$ ranges over all possible orientations $D$ of $H$. We will call $\ora{\gamma}_p(H)$ and $\ora{\Gamma}_p(H)$ the \emph{lower} and, respectively, the \emph{upper directed $p$-domination numbers.} An immediate observation is that 

\begin{equation}\label{mono}
\ora{\Gamma}_i(H) \le \ora{\Gamma}_j(H)
\end{equation}
for $1 \le i < j \le r-1$.\\

Back to the seminal result of Erd\H{o}s, it can be translated to
\[\log n - 2 \log\log n +1 \le \ora{\Gamma}_1(H(n,2)) \le \log(n+1),\]
where $H(n,r)$ stands for the complete $r$-uniform hypergraph on $n$ vertices. Surprisingly, even to prove that $\ora{\Gamma}_1(H(n,3))$ is unbounded as a function of $n$ (a problem raised by A. Gy\'arf\'as, see \cite{KoSu}) turns to be highly non trivial and it only has been solved recently \cite{KoSu} with a growth function which is poly-logarithmic in $\log^*(n)$.\\

Our first aim in this work is to prove the following theorem.\\

\noindent
{\bf Theorem \ref{tournaments}}
{\it Let $n \ge r$ and $r \ge 2$ be positive integers.
\begin{enumerate}
\item[(i)] For every integer $p$ with $1 \le p \le r-1$, 
\[\ora{\Gamma}_p(H(n,r))< r \left( 1 +\ln(n+(r-1)^2) \right).\]
\item[(ii)] There is a constant $c = c(r) > 0$ such that 
\[\ora{\Gamma}_{r-1}(H(n,r)) \ge c \; (\ln n)^{\frac{1}{r-1}}.\]
\end{enumerate}
}

The upper bound is proved by a standard application of the so called Greedy Partition Lemma (GPL),  which is developed in \cite{CaHe2} in full generality.  The lower bound is proved by the probabilistic method via, a bit involved yet, standard expectation argument.\\

Note that while, for $r = 2$, $p = 1$, the lower and upper bounds in Theorem \ref{tournaments} have the same order of magnitude (which is weaker than the nearly exact result of Erd\H{o}s),
already for $r = 3$ and $p = 2$ the most we get is $c \sqrt{\log n} \le \ora{\Gamma}_2(H(n,3)) \le C \log n$, for some other constant $C > 0$. Closing the gap for $r = 3$ and $p=2$ is an interesting problem in view of the much more dramatic gap in case $r = 3$ and $p = 1$.\\

Our second aim in this note is to generalize a theorem first proved in \cite{CaHe1} for directed graphs to the context of directed $r$-uniform hypergraphs using Theorem \ref{tournaments}. Given a hypergraph $H =(V, \mathcal{E})$, we call a subset $S \subseteq V(H)$ \emph{independent} if no edge is fully contained in $S$, that is, $E \setminus S \neq \emptyset$ for every edge $E \in \mathcal{E}$. We denote with $\alpha(H)$ the maximum cardinality of an independent set of $H$. A {\it proper coloring} of $H$ is a mapping $f: V \rightarrow \mathbb{N}$ such that no edge is monochromatic, that is, $|f(E)| \ge 2$ for every edge $E \in \mathcal{E}$. The minimum number of colors that is needed for a proper coloring of $H$ is denoted with $\chi(H)$. A \emph{clique} $S \subseteq V$ in $H$ is a set of vertices such that every $r$-set $E \subseteq S$ is an edge $E \in \mathcal{E}$. The cardinality of a clique of maximum size is called the \emph{clique number} $\omega(H)$ of $H$. Observe that, for the complementary hypergraph $\overline{H}$ of $H$, we have $\alpha(H) = \omega(\overline{H})$. Moreover, every proper coloring of $H$ with $k$ colors induces a partition of the vertex set into $k$ different color classes, each being an independent set of $H$. Hence, $\alpha(H) \ge \frac{n}{\chi(H)}$, where $n = |V|$.\\

\noindent
{\bf Theorem \ref{thm:chi}}
{\it Let $r \ge 2$} and $p$ be integers such that $1 \le p \le r-1$. Let $H$ be an $r$-uniform hypergraph. Then
\[\ora{\Gamma}_p(H) \le r\; \chi(\overline{H}) \left( 1 + \ln \left( \frac{n}{\chi(\overline{H})} + (r-1)^2 \right) \right).\]

\section{Further notation and definitions}

We shall now complete the notation used in this paper. Given a hypergraph $H  = (V,\mathcal{E})$, the number of vertices and edges of $H$ is denoted with $n(H)$ and $e(H)$, respectively. Let $E(v)$ be the set of edges that contains $v$ and $\deg(v) = |E(v)|$ the {\it degree} of $v$. 

We denote with $H(n,r)$ the complete $r$-uniform hypergraph on $n$ vertices. For a hypergraph $H =  (V,\mathcal{E})$ and a subset $U \subset V$, the {\it induced subhypergraph} $H[U]$ of $H$ by $U$ is the hyperpgraph with vertex set $U$ and all edges $E\in \mathcal{E}$ such that $E \subseteq U$.

Let $H = (V, \mathcal{E})$ be an $r$-uniform hypergraph and let $p$ be an integer such that $1 \le p \le r-1$. A set $S \subseteq V(H)$  is called {\it $p$-dominating set} if for every vertex $u \in V \setminus S$ there is an edge $E \in E(u)$ such that $|E \cap S| \ge p$. If $S$ is a $p$-dominating set of minimum cardinality, we call it a {\it minimum $p$-dominating set} and we set $\gamma_p(H) = |S|$. Note that $\gamma_1(H)$ coincides with the concept of domination studied in \cite{Ach1, Ach2, JoTu}.

Let $D$ be an orientation of $H$. For a set $A \subseteq V$ with $1 \le |A| \le r-1$, define $\ora{E}_D(A)$ as the set of edges  $E \in \mathcal{E}_H$ where  $A$ occupies the first $|A|$ positions in $E$ under orientation $D$ and $\ora{N}_D(A) = \cup_{E \in \ora{E}_D(A)} E \setminus A$. Let $\ora{\deg}_D(A) =  |\ora{E}_D(A)|$ and $\ora{n}_D(A) = |\ora{N}_D(A)|$. For simplicity, we set $\ora{E}_D(\{v\}) = \ora{E}_D(v)$, $\ora{N}_D(\{v\}) = \ora{N}_D(v)$, $\ora{\deg}_D(\{v\}) = \ora{\deg}_D(v)$ and $\ora{n}_D(\{v\}) = \ora{n}_D(v)$ for any vertex $v \in V$. We denote $\ora{\Delta}_p(D) = \max \ora{\deg}_D(A)$, among all sets $A \subseteq V$ with $|A| = p$, and $\ora{\Delta}_p(H) = \max \ora{\Delta}_p(D)$, where the maximum is taken among all possible orientations $D$ of $H$.

\section{Proofs of the theorems}

Before proving our main results, we need to formulate the so called {\it Greedy Patition Lemma} or, for short, {\it GPL}, that was given in \cite{CaHe2} and which will be the main tool in proving upper bounds on directed domination. Note that the authors of \cite{CaHe2} give a much more general version of this lemma, here we will state it just in the form that we will require along this paper. The GPL lemma as stated in \cite{CaHe2} is just one of several forms suitable to approximate through a greedy algorithm that in each step deletes certain subset of the vertex set of the considered hypergraph $H$ that satisfies a particular property. Here this subset consists of a set $A$ of cardinality $p$ and all the vertices directed dominated by $A$ in certain orientation $D$ of the considered hypergraph, namely $\ora{N}_D(A) \cup A$.  The crucial point in this lemma and its variations is that,  in each step, we approximate from below the order of the deleted subset by a function $f(x)$ where $x$ is the cardinality of the current structure (from which then this vertex subset is deleted).  So adaptation of the GPL version from \cite{CaHe2} is simply done by observing that, for $x \ge 2r – 1$, the directed dominating set we remove at each step has cardinality at least $(x –r +1)/r + p$.

\begin{theorem}{\bf(Greedy Partition Lemma (GPL), \cite{CaHe2})}\label{GPL}
Let $\mathcal{H}$ be a class of hypergraphs closed under induced subhypergraphs. Let $t \ge 2$ be an integer and let $f:[t, \infty) \rightarrow [1,\infty)$ be a positive nondecreasing continuous function. If for any hypergraph $H \in \mathcal{H}$ and any orientation $D$ of $H$, we have that
\[\max_{A \subseteq V(H), |A|=p} n_D(A) + p \ge f(|V(H)|),\]
then
\[\ora{\Gamma}_p(H) \le t + \int_{t}^{\max\{|V(H)|,t\}} \frac{1}{f(x)} dx.\]
\end{theorem}

We can now prove our main result.

\begin{theorem}\label{tournaments}
Let $n \ge r$ and $r \ge 2$ be positive integers.
\begin{enumerate}
\item[(i)] For every integer $p$ with $1 \le p \le r-1$, 
\[\ora{\Gamma}_p(H(n,r)) < r \left( 1 +\ln(n+(r-1)^2) \right).\]
\item[(ii)] There is a constant $c = c(r) > 0$ such that 
\[\ora{\Gamma}_{r-1}(H(n,r)) \ge c \; (\ln n)^{\frac{1}{r-1}}.\]
\end{enumerate}
\end{theorem}

\begin{proof}
(i) For proving the upper bound, observe that, because of the monotonicity \eqref{mono}, it suffices to prove the theorem for $p = r-1$.

We take an orientation $D$ of the edges of $H(n,r)$.
Let $A_{\Delta} \subseteq V(H(n,r))$ be a set of $r-1$ vertices such that $\ora{\deg}_D(A_{\Delta}) = \ora{\Delta}_{r-1}(D)$ and observe that
\[\ora{n}_D(A_{\Delta}) \ge \ora{\deg}_D(A_{\Delta}) = \ora{\Delta}_{r-1}(D).\]
On the other hand, since $\displaystyle\sum_{A \subseteq V(D), |A| = r-1} \ora{\deg}_D(A) = \binom{n}{r}$, we have
\begin{align*}
\ora{\Delta}_{r-1}(D) &\ge {\binom{n}{r-1}}^{-1} \displaystyle\sum_{A \subseteq V(D), |A| = r-1} \ora{\deg}_D(A)\\
&= {\binom{n}{r-1}}^{-1} \binom{n}{r} \\
&= \frac{(r-1)! (n-r+1)!}{r! (n-r)!} \\
& = \frac{n-r+1}{r}.
\end{align*}

Hence, combining both inequalities, we obtain  $\ora{n}_D(A_{\Delta}) \ge \frac{n-r+1}{r}.$

Setting $f(x) = \frac{x+(r-1)^2}{r}$, we have a non-decreasing function $f$ with $f(x) \ge 1$ for $x \ge 2r-1$ and such that 
\[f(n) = \frac{n+(r-1)^2}{r} = \frac{n-r+1}{r} + r-1  \le  \ora{n}_D(A_{\Delta}) + r-1.\]
Thus, we can apply the GPL, which leads to
\begin{align*}
\ora{\Gamma}_{r-1}(H) &\le  2r-1 + \int_{2r-1}^n  \frac{r}{x+(r-1)^2}dx \\
&= 2r-1 + r\ln(n+(r-1)^2) - r \ln(2r-1 +(r-1)^2) \\
&= 2r-1 + r\ln(n+(r-1)^2) - r \ln(r^2) \\
&= r \left(2-\frac{1}{r} + \ln(n+(r-1)^2) -  \ln(r^2) \right)\\
& = r \left( 1 +\ln(n+(r-1)^2) + \ln \left( \frac{e^{1 - \frac{1}{r}}}{r^2} \right)\right)
\end{align*}
Finally, noting that $e^{1-\frac{1}{r}} < r^2$ for $r \ge 2$, we obtain
\[
\ora{\Gamma}_{r-1}(H)<  r \left( 1 +\ln(n+(r-1)^2) \right).\]

\noindent
(ii) Let $D$ be a random orientation of the edges of $H$ such that every edge is given independently one of the possible $r!$ linear orders. Let $S \subseteq V$ be a set of $t \ge r-1$ vertices and $v \in V \setminus S$ a vertex not in $S$. Consider a fixed set $A \subseteq V$ of $r-1$ vertices. Then the probability that $A$ occupies the first $r-1$ positions in an edge $E \in \mathcal{E}(H)$, say $E = A \cup \{v\}$, under orientation $D$ is equal to $\frac{(r-1)!}{r!} = \frac{1}{r}$, namely the number of linear orders of the edge $E = A \cup \{v\}$ where $v$ appears last divided by the number of all possible linear orders of $E = A \cup \{v\}$. Thus the probability that $A$ does not $(r-1)$-dominate $v$ is $1- \frac{1}{r} = \frac{r-1}{r}$. This implies that the probability that there is a set $A \subseteq S$ with $|A| = r-1$ that $(r-1)$-dominates $v$ is equal to $1- (\frac{r-1}{r})^{\binom{t}{r-1}}$, namely $1$ minus the probability that no $(r-1)$-subset of $S$ $(r-1)$-dominates $v$. Therefore, the probability that $S$ $(r-1)$-dominates all vertices in $V \setminus S$ is equal to 
\[\left(1- \left(\frac{r-1}{r}\right)^{\binom{t}{r-1}}\right)^{n-t}.\]
Let $x$ be the number of directed $(r-1)$-dominating sets of cardinality $t$ in $H$ under orientation $D$. Then we have
\[E[x] = \binom{n}{t} \left(1- \left( \frac{r-1}{r}\right)^{\binom{t}{r-1}}\right)^{n-t}.\]
Note that, if $E[x] < 1$, then there is an orientation of $H$ such that there is no directed $(r-1)$-domination set of cardinality $t$. We will now determine the best possible $t$, i.e. we will try to find a $t$ as large as possible such that $E[x] < 1$. Since
\[ E[x] = \binom{n}{t} \left(1- \left(1 - \frac{1}{r}\right)^{\binom{t}{r-1}}\right)^{n-t} 
         < \left(\frac{n \cdot e}{t}\right)^t e^{-(n-t) (\frac{r-1}{r})^{\binom{t}{r-1}}},\]
it is sufficient to solve
\[\left(\frac{n \cdot e}{t}\right)^t < e^{(n-t) (\frac{r-1}{r})^{\binom{t}{r-1}}},\]
or, equivalently,
\[t \cdot \ln n + t - t \ln t < (n-t) \left(\frac{r-1}{r}\right)^{\binom{t}{r-1}}\]
or rather 
\[(t \cdot \ln n + t - t \ln t) \left(\frac{r}{r-1}\right)^{\binom{t}{r-1}} +t < n.\]
Since $2 < \ln t$ for $t > 7$, which we may assume for $n$ large enough,
\[(t \cdot \ln n + t - t \ln t) \left(\frac{r}{r-1}\right) ^{\binom{t}{r-1}} +t < (t \cdot \ln n + 2t - t \ln t) \left(\frac{r}{r-1}\right)^{\binom{t}{r-1}}.\] 
Hence, it is sufficient to solve
\[t \ln n \left(\frac{r}{r-1}\right)^{\binom{t}{r-1}} < n.\]
Moreover, since $\binom{t}{r-1} < (\frac{t \cdot e}{r-1})^{r-1}$, it will be enough to solve the inequality 
\[t \ln n (\frac{r}{r-1})^{(\frac{t \cdot e}{r-1})^{(r-1)}} < n,\]
which is equivalent to
\[\ln t + \ln(\ln n) + \left(\frac{t \cdot e}{r-1}\right)^{(r-1)} \ln\left(\frac{r}{r-1}\right) < \ln n.\]
From the upper bound given in (i), we know that $t \le  r + r\ln(n-r+1) = r(1 + \ln(n-r+1))$. Hence,
\begin{align*}
&\displaystyle \ln t + \ln(\ln n) + \left(\frac{t \cdot e}{r-1}\right)^{(r-1)} \ln\left(\frac{r}{r-1}\right)\\
& \le  \displaystyle \ln(r(1 + \ln(n-r+1))) + \left(\frac{t \cdot e}{r-1}\right)^{(r-1)} \ln\left(\frac{r}{r-1}\right).
\end{align*}
We may assume $\ln(r(1 + \ln(n-r+1))) + \ln(\ln n) < \frac{\ln n}{2}$ for $n$ large enough, thus it is sufficient to solve
\[\frac{\ln n}{2} +  \left(\frac{t \cdot e}{r-1}\right)^{(r-1)} \ln\left(\frac{r}{r-1}\right) < \ln n,\]
which leads finally to
\[t < \left( \frac{\ln n}{2 \ln(\frac{r}{r-1})}\right)^{\frac{1}{r-1}} \frac{r-1}{e} = c \;(\ln n)^{\frac{1}{r-1}}\]
for a constant $c = c(r) > 0$ depending on $r$. Altogether it follows that
\[\ora{\Gamma}_{r-1}(H) \ge c \;(\ln n)^{\frac{1}{r-1}}.\]

\end{proof}

We shall now use Theorem \ref{tournaments}(i) to get a general upper bound for arbitrary $r$-uniform hypergraphs. However, we first need a lemma.

\begin{lemma}\label{decompo}
 Let $H$ be a hypergraph and let $V_1, V_2, \ldots, V_k$ be subsets of $V$, not necessarily disjoint, such that $\cup_{i=1}^k V_i = V$. Let $H_i = H[V_i]$. Then
 \[\ora{\Gamma}_p(H) \le \sum_{i=1}^k \ora{\Gamma}_p(H_i).\]
 \end{lemma}

\begin{proof}
Consider an arbitrary orientation $D$ of $H$. For each $i = 1, 2, \ldots, k$, let $D_i$ be the orientation induced by $D$ on $H_i$ and let $S_i$ be a directed $p$-dominating set of $D_i$ of minimum cardinality. Evidently $\ora{\Gamma}_p(H_i) \ge \ora{\gamma}_p(D_i) = |S_i|$. Moreover, $S = \cup_{i=1}^k S_i$ is a directed $p$-dominating set of $D$. This implies
\[\ora{\gamma}_p(D) \le \sum_{i=1}^k |S_i| = \sum_{i=1}^k \ora{\gamma}_p(D_i) \le \sum_{i=1}^k \ora{\Gamma}_p(H_i).\]
 Since $D$ was chosen arbitrarily, the inequality holds also for $\ora{\Gamma}_p(H)$ and thus
\[\ora{\Gamma}_p(H) \le \sum_{i=1}^k \ora{\Gamma}_p(H_i).\]
\end{proof}

Now we can state and prove our second main result. In \cite{CaHe1}, Caro and Henning prove that, for a graph $G$ of order $n$,
\[\ora{\Gamma}_p(G) \le  \chi(\overline{G})\; \log \left( \frac{n}{\chi(\overline{G})} +1\right)\]
holds. Inspired by this result, we give in the next theorem a similar statement for hypergraphs. 

\begin{theorem}\label{thm:chi}
Let $r \ge 2$ and $p$ be integers such that $1 \le p \le r-1$. Let $H$ be an $r$-uniform hypergraph. Then
\[\ora{\Gamma}_p(H) \le r\; \chi(\overline{H}) \left( 1 + \ln \left( \frac{n}{\chi(\overline{H})} + (r-1)^2 \right) \right).\]
\end{theorem}

\begin{proof}
Let $t = \chi(\overline{H})$ and consider a proper coloring of $\overline{H}$ into $t$ color classes $Q_1, Q_2, \ldots, Q_t$ with $|Q_i| = q_i$ for $1 \le i \le t$. Since every color class is an independent set, we observe that the sets $Q_1, Q_2, \ldots, Q_t$ correspond to  cliques in $H$ or to trivial sets if $1 \le q_i \le r-1$.

Let $D$ be an arbitrary orientation of $H$ and let $D_i = D[Q_i]$ denote the orientation of $H[Q_i]$ induced by $D$. 
Then, by Lemma \ref{decompo},
\[\ora{\gamma}_p(D) \le \sum_{i=1}^t \ora{\gamma}_p(D_i)  \le \sum_{i=1}^t  \ora{\Gamma}_p(H[Q_i])  = \sum_{i=1}^t \ora{\Gamma}_p(H(q_i,r)).\]
If $q_i \le r-1$ for some $i \in \{1, 2, \ldots, t\}$, then 
\[\ora{\Gamma}_p(H(q_i,r)) = q_i < r \le r \left( 1 +\ln(q_i+(r-1)^2) \right).\]
On the other hand, if $q_i \ge r$, then Theorem \ref{tournaments} yields
\[\ora{\Gamma}_p(H(q_i,r)) <  r \left( 1 +\ln(q_i+(r-1)^2) \right).\]
Hence, since $D$ was arbitrarily chosen and $\ora{\Gamma}_p(H(q_i,r)) = q_i \le  r \left( 1 +\ln(q_i+(r-1)^2) \right)$ for all $1 \le i \le t$ by the above discussion, we obtain
\[\ora{\Gamma}_p(H) \le \sum_{i=1}^t \ora{\Gamma}_p(H(q_i,r))  \le \sum_{i=1}^t  r \left( 1 +\ln(q_i+(r-1)^2) \right).\]
Since $n =  \sum_{i=1}^t q_i$, it follows, by Jensen's inequality applied on the concave function $\ln x$, that
\begin{align*}
\ora{\Gamma}_p(H) &\le \sum_{i=1}^t   r \left( 1 +\ln(q_i+(r-1)^2) \right)\\
&= rt + r  \sum_{i=1}^t  \ln(q_i+(r-1)^2) \\
&\le rt + rt \ln \left( \frac{n + t (r-1)^2}{t}\right) \\
&= rt \left( 1 + \ln \left( \frac{n}{t} + (r-1)^2\right) \right).
\end{align*}
Hence, we have proved that
\[\ora{\Gamma}_p(H) \le r\; \chi(\overline{H}) \left( 1 + \ln \left( \frac{n}{\chi(\overline{H})}+ (r-1)^2 \right) \right).\]
\end{proof}

Using GPL has the advantage of getting at once a general upper bound for every choice of $n$ and $r$  for the directed domination number of $H(n,r)$. The disadvantage of using GPL is that it is known not to give the best possible leading constant. Still, for our main purpose, which is to get the logarithmic upper bound in Theorem \ref{tournaments}~(i) for every choice of $r$ and $n$, GPL suffices. The method also leads to slightly weaker results in Theorem \ref{thm:chi} for the case $r = 2$ in comparison with the bound in \cite{CaHe1} derived directly from Erd\H{o}s' upper bound for the case $r = 2$. However, getting better leading constants already for $r \ge 3$ seems less important as long as the order of magnitude of even  $\ora{\Gamma}_{r-1}(H(n,r))$ is not known.\\

\section{Open problems}
We close this paper with a few problems.

\begin{problem}
Certainly the most challenging problem is to try to determine the correct order of $\ora{\Gamma}_p(H(n,r))$  for $r \ge 3$ and $1 \le p \le r-1$.
\end{problem} 

\begin{problem}
Is it true that $\ora{\Gamma}_2(H(n,3))  =  \Theta(\sqrt{\log n})$?
\end{problem} 

\begin{problem}
Is it true that, for some positive constants $c$ and $\alpha$, $\ora{\Gamma}_1(H(n,3)) \ge  c (\log n)^\alpha$?
\end{problem} 

\begin{problem}
Let $r$, $p$ be positive integers such that $1 \le p \le r-1$. Is it true that, for an arbitrary $r$-uniform hypergaph $H$ of order $n$, there is a constant $c(r,p) > 0$ such that $\ora{\Gamma}_p(H) \le  c(r,p) \alpha(H) \ln n$?
\end{problem} 

\section*{Dedication}
This paper is dedicated to Lutz Volkmann on the occasion of his 75th birthday. He was my (A.H.) Ph.D. advisor and my formation and success  as a combinatorialist are deeply grounded in what he taught me, for which I am profoundly thankful to him. Since my work with Volkmann is centered on domination and he is also very fond of directed graphs, specially tournaments, we chose this particular topic for him.

\section*{Acknowledgments}
We would like to thank the anonymous referee who helped improve the presentation of this paper.\\
The second author was partially supported by PAPIIT IA103217, PAPIIT IN111819, and CONACyT project 219775.

\end{document}